\documentclass[centertags,reqno,a4paper,fleqn]{amsart}
\usepackage{amsfonts}

\newtheorem{theorem}{Theorem}
\theoremstyle{plain}

\newtheorem{corollary}{Corollary}

\newtheorem{definition}{Definition}

\newtheorem{lemma}{Lemma}

\newtheorem{remark}{Remark}

\numberwithin{equation}{section}
\input{tcilatex}

\begin{document}
\title[]{Hermite-Hadamard type inequalities for harmonically $\left( \alpha
,m\right) $-convex functions via fractional integrals}
\author{Mehmet Kunt}
\address{Department of Mathematics, Faculty of Sciences, Karadeniz Technical
University, Trabzon, Turkey}
\email{mkunt@ktu.edu.tr}
\author{\.{I}mdat \.{I}\c{s}can}
\address{Department of Mathematics, Faculty of Sciences and Arts, Giresun
University, Giresun, Turkey}
\email{imdat.iscan@giresun.edu.tr}

\begin{abstract}
In this paper, some Hermite-Hadamard type inequalities are established for
harmonically $\left(\alpha,m\right)$-convex functions via fractional
integrals and some Hermite-Hadamard type inequalities are obtained for these
classes of functions.
\end{abstract}

\maketitle

\section{introduction}

Let $f:I\subset\mathbb{R}\rightarrow\mathbb{R}$ be a convex function defined
on the interval $I$ of real numbers and $a,b\in I$ with $a<b$. The following
inequality is well known in the literature as Hermite-Hadamard integral
inequality for convex functions

\begin{equation}
f\left( \frac{a+b}{2}\right) \leq \frac{1}{b-a}\int_{a}^{b}f\left( x\right)
dx\leq \frac{f\left( a\right) +f\left( b\right) }{2}.  \label{1.1}
\end{equation}

Both inequalities hold in the reversed direction if $f$ is concave. Note
that, some of the classical inequalities for means can be obtained from
appropriate particular selections of the mapping $f$. For some results which
generalize, improve and extend the inequalities $\left( \ref{1.1}\right) $
we refer the reader to the recent paper \cite{[1]}-\cite{[5]} and references
therein.

In \cite{[1]}, \.{I}\c{s}can gave definition of harmonically convex
functions and established some Hermite-Hadamard type inequalities for
harmonically convex functions as follows:

\begin{definition}
\bigskip Let $I\subset 
%TCIMACRO{\U{211d} }%
%BeginExpansion
\mathbb{R}
%EndExpansion
\backslash \left\{ 0\right\} $ be a \ real interval. A function $%
f:I\rightarrow 
%TCIMACRO{\U{211d} }%
%BeginExpansion
\mathbb{R}
%EndExpansion
$ is said to be harmonically convex, if%
\begin{equation}
f\left( \frac{xy}{tx+\left( 1-t\right) y}\right) \leq tf\left( y\right)
+\left( 1-t\right) f\left( x\right)  \label{1.2}
\end{equation}%
for all $x,y\in I$ and $t\in \left[ 0,1\right] $. If the inequality in $%
\left( \ref{1.2}\right) $ is reversed, then $f$ is said to be harmonically
concave.
\end{definition}

\begin{theorem}
\cite{[1]}. Let $f:I\subset 
%TCIMACRO{\U{211d} }%
%BeginExpansion
\mathbb{R}
%EndExpansion
\backslash \left\{ 0\right\} \rightarrow 
%TCIMACRO{\U{211d} }%
%BeginExpansion
\mathbb{R}
%EndExpansion
$ be a harmonically convex function and $a,b\in I$ with $a<b$. If $f\in L%
\left[ a,b\right] $ then the following inequalities hold: 
\begin{equation*}
f\left( \frac{2ab}{a+b}\right) \leq \frac{ab}{b-a}\int_{a}^{b}\frac{f\left(
x\right) }{x^{2}}dx\leq \frac{f\left( a\right) +f\left( b\right) }{2}.
\end{equation*}
\end{theorem}

\begin{theorem}
\cite{[1]}. Let $f:I\subset \left( 0,\infty \right) \rightarrow 
%TCIMACRO{\U{211d} }%
%BeginExpansion
\mathbb{R}
%EndExpansion
$ be a differentiable function on $I%
%TCIMACRO{\U{b0}}%
%BeginExpansion
{{}^\circ}%
%EndExpansion
$, $a,b\in I$ with $a<b,$ and $f^{\prime }\in L\left[ a,b\right] $. If $%
\left\vert f^{\prime }\right\vert ^{q}$ is harmonically convex on $\left[ a,b%
\right] $ for $q\geq 1$, then%
\begin{equation*}
\left\vert \frac{f\left( a\right) +f\left( b\right) }{2}-\frac{ab}{b-a}%
\int_{a}^{b}\frac{f\left( x\right) }{x^{2}}dx\right\vert
\end{equation*}%
\begin{equation}
\leq \frac{ab\left( b-a\right) }{2}\lambda _{1}^{1-1/q}\left[ \lambda
_{2}\left\vert f^{\prime }\left( a\right) \right\vert ^{q}+\lambda
_{3}\left\vert f^{\prime }\left( b\right) \right\vert ^{q}\right] ^{1/q},
\label{1.3}
\end{equation}%
where%
\begin{eqnarray*}
\lambda _{1} &=&\frac{1}{ab}-\frac{2}{\left( b-a\right) ^{2}}\ln \left( 
\frac{\left( a+b\right) ^{2}}{4ab}\right) , \\
\lambda _{2} &=&\frac{-1}{b\left( b-a\right) }+\frac{3a+b}{\left( b-a\right)
^{3}}\ln \left( \frac{\left( a+b\right) ^{2}}{4ab}\right) , \\
\lambda _{3} &=&\frac{1}{a\left( b-a\right) }-\frac{3b+a}{\left( b-a\right)
^{3}}\ln \left( \frac{\left( a+b\right) ^{2}}{4ab}\right) =\lambda
_{1}-\lambda _{2}.
\end{eqnarray*}
\end{theorem}

\begin{theorem}
\cite{[1]}. Let $f:I\subset \left( 0,\infty \right) \rightarrow 
%TCIMACRO{\U{211d} }%
%BeginExpansion
\mathbb{R}
%EndExpansion
$ be a differentiable function on $I%
%TCIMACRO{\U{b0}}%
%BeginExpansion
{{}^\circ}%
%EndExpansion
$, $a,b\in I$ with $a<b,$ and $f^{\prime }\in L\left[ a,b\right] $. If $%
\left\vert f^{\prime }\right\vert ^{q}$ is harmonically convex on $\left[ a,b%
\right] $ for $q>1$, $\frac{1}{p}+\frac{1}{q}=1,$ then%
\begin{equation*}
\left\vert \frac{f\left( a\right) +f\left( b\right) }{2}-\frac{ab}{b-a}%
\int_{a}^{b}\frac{f\left( x\right) }{x^{2}}dx\right\vert
\end{equation*}%
\begin{equation}
\leq \frac{ab\left( b-a\right) }{2}\left( \frac{1}{p+1}\right) ^{1/p}\left[
\mu _{1}\left\vert f^{\prime }\left( a\right) \right\vert ^{q}+\mu
_{2}\left\vert f^{\prime }\left( b\right) \right\vert ^{q}\right] ^{1/q},
\label{1.4}
\end{equation}%
where%
\begin{eqnarray*}
\mu _{1} &=&\frac{\left[ a^{2-2q}+b^{1-2q}\left[ \left( b-a\right) \left(
1-2q\right) -a\right] \right] }{2\left( b-a\right) ^{2}\left( 1-q\right)
\left( 1-2q\right) }, \\
\mu _{2} &=&\frac{\left[ b^{2-2q}-a^{1-2q}\left[ \left( b-a\right) \left(
1-2q\right) +b\right] \right] }{2\left( b-a\right) ^{2}\left( 1-q\right)
\left( 1-2q\right) }.
\end{eqnarray*}
\end{theorem}

In \cite{[8]}, Miha\c{s}en gave definition of $\left( \alpha ,m\right) $%
-convex functions as follows:

\begin{definition}
The function $f:\left[ 0,b\right] \rightarrow 
%TCIMACRO{\U{211d} }%
%BeginExpansion
\mathbb{R}
%EndExpansion
$, $b>0$, is said to be $\left( \alpha ,m\right) $-convex where $\left(
\alpha ,m\right) \in \left[ 0,1\right] ^{2}$, if we have%
\begin{equation*}
f\left( tx+m\left( 1-t\right) y\right) \leq t^{\alpha }f\left( x\right)
+m\left( 1-t^{\alpha }\right) f\left( y\right)
\end{equation*}%
for all $x,y\in \left[ 0,b\right] $ and $t\in \left[ 0,1\right] $.
\end{definition}

It can be easily that for $\left( \alpha ,m\right) \in \left\{ \left(
0,0\right) ,\left( \alpha ,0\right) ,\left( 1,0\right) ,\left( 1,m\right)
,\left( 1,1\right) ,\left( \alpha ,1\right) \right\} $ one obtains the
following classes of functions: increasing, $\alpha $-starshaped,
starshaped, $m$-convex, convex, $\alpha $-convex.

For recent results and generalizations concerning $\left( \alpha ,m\right) $%
-convex functions we refer the reader to paper \cite{[8]}-\cite{[12]} and
references therein.

In \cite{[6]}, \.{I}\c{s}can gave definition of harmonically $\left( \alpha
,m\right) $-convex functions as follows:

\begin{definition}
The function $f:\left( 0,b^{\ast }\right] \rightarrow 
%TCIMACRO{\U{211d} }%
%BeginExpansion
\mathbb{R}
%EndExpansion
$, $b^{\ast }>0$, is said to be harmonically $\left( \alpha ,m\right) $%
-convex, where $\alpha \in \left[ 0,1\right] $ and $m\in \left( 0,1\right] $%
, if%
\begin{equation}
f\left( \frac{mxy}{mty+\left( 1-t\right) x}\right) \leq t^{\alpha }f\left(
x\right) +m\left( 1-t^{\alpha }\right) f\left( y\right)  \label{1.5}
\end{equation}%
for all $x,y\in \left( 0,b^{\ast }\right] $ and $t\in \left[ 0,1\right] $.
If the inequality in $\left( \ref{1.5}\right) $ is reversed, then $f$ is
said to be harmonically $\left( \alpha ,m\right) $-concave.
\end{definition}

Note that $\left( \alpha ,m\right) \in \left\{ \left( 1,m\right) ,\left(
1,1\right) ,\left( \alpha ,1\right) \right\} $ one obtains the following
classes of functions: harmonically $m$-convex, harmonically convex,
harmonically $\alpha $-convex (or harmonically $s$-convex in the first
sense, if we take $s$ instead of $\alpha $).

We recall the following inequality and special functions which are known as
Beta and hypergeometric function respectively%
\begin{equation*}
\beta \left( x,y\right) =\frac{\Gamma \left( x\right) \Gamma \left( y\right) 
}{\Gamma \left( x+y\right) }=\int_{0}^{1}t^{x-1}\left( 1-t\right) ^{y-1}dt,%
\text{ }x,y>0,
\end{equation*}

\begin{eqnarray*}
_{2}F_{1}\left( a,b;c;z\right) &=&\frac{1}{\beta \left( b,c-b\right) }%
\int_{0}^{1}t^{b-1}\left( 1-t\right) ^{c-b-1}\left( 1-zt\right) ^{-a}dt,%
\text{ } \\
c &>&b>0,\left\vert z\right\vert <1\left( \text{see \cite{[13]}}\right) .
\end{eqnarray*}

\begin{lemma}
\cite{[14]}, \cite{[15]}. For $0<\theta \leq 1$ and $0\leq a<b$ we have%
\begin{equation*}
\left\vert a^{\theta }-b^{\theta }\right\vert \leq \left( b-a\right)
^{\theta }.
\end{equation*}
\end{lemma}

Following definitions and mathematical preliminaries of fractional calculus
theory are used further in this paper.

\begin{definition}
Let $f\in L\left[ a,b\right] $. The Riemann-Liouville integrals $%
J_{a+}^{\theta }f$ and $J_{b-}^{\theta }f$ of order $\theta >0$ with $a\geq
0 $ are defined by%
\begin{equation*}
J_{a+}^{\theta }f\left( x\right) =\frac{1}{\Gamma \left( \theta \right) }%
\int_{a}^{x}\left( x-t\right) ^{\theta -1}f\left( t\right) dt,\text{ }x>a
\end{equation*}%
and 
\begin{equation*}
J_{b-}^{\theta }f\left( x\right) =\frac{1}{\Gamma \left( \theta \right) }%
\int_{x}^{b}\left( t-x\right) ^{\theta -1}f\left( t\right) dt,\text{ }x<b
\end{equation*}%
respectively, where $\Gamma $ is the Euler Gamma function defined by $\Gamma
\left( \theta \right) =\int_{0}^{\infty }e^{-t}t^{\theta -1}dt$ and $%
J_{a+}^{0}f\left( x\right) =J_{b-}^{0}f$ $\left( x\right) =f\left( x\right) $%
.
\end{definition}

Let $f:I\subset \left( 0,\infty \right) \rightarrow 
%TCIMACRO{\U{211d} }%
%BeginExpansion
\mathbb{R}
%EndExpansion
$ be a differentiable function on $I%
%TCIMACRO{\U{b0}}%
%BeginExpansion
{{}^\circ}%
%EndExpansion
$, throughout this paper we will take

\begin{eqnarray*}
I_{f}\left( g;\theta ,a,b\right) &=&\frac{f\left( a\right) +f\left( b\right) 
}{2}-\frac{\Gamma \left( \theta +1\right) }{2}\left( \frac{ab}{b-a}\right)
^{\theta } \\
&&\times \left\{ J_{1/a-}^{\theta }\left( f\circ g\right) \left( 1/b\right)
+J_{1/b+}^{\theta }\left( f\circ g\right) \left( 1/a\right) \right\} .
\end{eqnarray*}%
where $a,b\in I$ with $a<b$, $\theta >0$, $g\left( x\right) =1/x$.

In \cite{[7]}, the authors represented Hermite-Hadamard's inequalities for
harmonically convex functions in fractional integral forms as follows:

\begin{theorem}
Let $f:I\subset \left( 0,\infty \right) \rightarrow 
%TCIMACRO{\U{211d} }%
%BeginExpansion
\mathbb{R}
%EndExpansion
$ be a function such that $f\in L\left[ a,b\right] $, where $a,b\in I$ with $%
a<b$. If $f$ is a harmonically convex function on $\left[ a,b\right] $, then
the following inequalities for fractional integrals hold:%
\begin{equation*}
f\left( \frac{2ab}{a+b}\right) \leq \frac{\Gamma \left( \theta +1\right) }{2}%
\left( \frac{ab}{b-a}\right) ^{\theta }\left\{ 
\begin{array}{c}
J_{1/a-}^{\theta }\left( f\circ g\right) \left( 1/b\right) \\ 
+J_{1/b+}^{\theta }\left( f\circ g\right) \left( 1/a\right)%
\end{array}%
\right\} \leq \frac{f\left( a\right) +f\left( b\right) }{2}
\end{equation*}%
with $\theta >0$.
\end{theorem}

In \cite{[7]}, the authors gave the following identity for differentiable
functions.

\begin{lemma}
Let $f:I\subset \left( 0,\infty \right) \rightarrow 
%TCIMACRO{\U{211d} }%
%BeginExpansion
\mathbb{R}
%EndExpansion
$ be a differentiable function on $I%
%TCIMACRO{\U{b0}}%
%BeginExpansion
{{}^\circ}%
%EndExpansion
$ such that $f^{\prime }\in L\left[ a,b\right] $, where $a,b\in I$ with $a<b$%
. Then the following equality for fractional integrals holds:%
\begin{equation}
I_{f}\left( g;\theta ,a,b\right) =\frac{ab\left( b-a\right) }{2}\int_{0}^{1}%
\frac{t^{\theta }-\left( 1-t\right) ^{\theta }}{\left( ta+\left( 1-t\right)
b\right) ^{2}}f^{\prime }\left( \frac{ab}{ta+\left( 1-t\right) b}\right) dt
\label{1.6}
\end{equation}

\begin{remark}
The identity $\left( \ref{1.6}\right) $ is equal the following one%
\begin{equation}
I_{f}\left( g;\theta ,a,b\right) =\frac{ab\left( b-a\right) }{2}\int_{0}^{1}%
\frac{\left( 1-t\right) ^{\theta }-t^{\theta }}{\left( tb+\left( 1-t\right)
a\right) ^{2}}f^{\prime }\left( \frac{ab}{tb+\left( 1-t\right) a}\right) dt.
\label{1.7}
\end{equation}
\end{remark}
\end{lemma}

Because of the wide application of Hermite-Hadamard type inequalities and
fractional integrals, many researchers extend their studies to
Hermite-Hadamard type inequalities involving fractional integrals. Recent
results for this area, we refer the reader to paper \cite{[7]}, \cite{[15]}-%
\cite{[18]} and references therein.

In this paper, we aimed to establish Hermite-Hadamard's inequalities for
harmonically $\left( \alpha ,m\right) $-convex functions via fractional
integrals. These results have some relations with \cite{[1]}.

\section{main results}

\begin{theorem}
Let $f:I\subset \left( 0,\infty \right) \rightarrow \mathbb{R}$ be a
differentiable function on $I{{}%
%TCIMACRO{\U{b0}}%
%BeginExpansion
{{}^\circ}%
%EndExpansion
}$, $a,b/m\in I{{}%
%TCIMACRO{\U{b0}}%
%BeginExpansion
{{}^\circ}%
%EndExpansion
}$ with $a<b$, $m\in \left( 0,1\right] $ and $f^{\prime }\in L\left[ a,b%
\right] $. If $\left\vert f^{\prime }\right\vert ^{q}$ is harmonically $%
\left( \alpha ,m\right) $-convex on $\left[ a,b/m\right] $ for some fixed $\
q\geq 1$,with $\alpha \in \left[ 0,1\right] $, then 
\begin{eqnarray}
\left\vert I_{f}\left( g;\theta ,a,b\right) \right\vert &\leq &\frac{%
ab\left( b-a\right) }{2}C_{1}^{1-1/q}\left( \theta ;a,b\right)  \notag \\
&&\times \left[ C_{2}\left( \theta ;\alpha ;a,b\right) \left\vert f^{\prime
}\left( a\right) \right\vert ^{q}+mC_{3}\left( \theta ;\alpha ;a,b\right)
\left\vert f^{\prime }\left( b/m\right) \right\vert ^{q}\right] ^{1/q}\text{
\ \ \ \ }  \label{2.1}
\end{eqnarray}%
where%
\begin{equation*}
C_{1}\left( \theta ;a,b\right) =\frac{b^{-2}}{\theta +1}\left[ 
\begin{array}{c}
\begin{array}{c}
_{2}F_{1}\left( 2,\theta +1;\theta +2;1-\frac{a}{b}\right)%
\end{array}
\\ 
+%
\begin{array}{c}
_{2}F_{1}\left( 2,1;\theta +2;1-\frac{a}{b}\right)%
\end{array}%
\end{array}%
\right] \text{,}
\end{equation*}%
\begin{equation*}
C_{2}\left( \theta ;\alpha ;a,b\right) =\left[ 
\begin{array}{c}
\frac{\beta \left( \theta +1,\alpha +1\right) }{b^{2}}%
\begin{array}{c}
_{2}F_{1}\left( 2,\theta +1;\theta +\alpha +2;1-\frac{a}{b}\right)%
\end{array}
\\ 
+\frac{b^{-2}}{\theta +\alpha +1}%
\begin{array}{c}
_{2}F_{1}\left( 2,1;\theta +\alpha +2;1-\frac{a}{b}\right)%
\end{array}%
\end{array}%
\right] \text{,}
\end{equation*}%
\begin{equation*}
C_{3}\left( \theta ;\alpha ;a,b\right) =C_{1}\left( \theta ;a,b\right)
-C_{2}\left( \theta ;\alpha ;a,b\right) \text{.}
\end{equation*}

\begin{proof}
Let $A_{t}=tb+\left( 1-t\right) a$, $B_{u}=ua+\left( 1-u\right) b$. Since $%
\left\vert f^{\prime }\right\vert ^{q}$ is harmonically $\left( \alpha
,m\right) $-convex, using $\left( \ref{1.5}\right) $%
\begin{eqnarray}
\left\vert f^{\prime }\left( \frac{ab}{A_{t}}\right) \right\vert ^{q}
&=&\left\vert f^{\prime }\left( \frac{ab}{tb+\left( 1-t\right) a}\right)
\right\vert ^{q}=\left\vert f^{\prime }\left( \frac{ma\left( b/m\right) }{%
mt\left( b/m\right) +\left( 1-t\right) a}\right) \right\vert ^{q}  \notag \\
&\leq &t^{\alpha }\left\vert f^{\prime }\left( a\right) \right\vert
^{q}+m\left( 1-t^{\alpha }\right) \left\vert f^{\prime }\left( b/m\right)
\right\vert ^{q}\text{.}  \label{2.2}
\end{eqnarray}%
\qquad \qquad

From $\left( \ref{1.7}\right) $, using the property of the modulus, the
power mean inequality and $\left( \ref{2.2}\right) $, we find%
\begin{eqnarray}
\left\vert I_{f}\left( g;\theta ,a,b\right) \right\vert &\leq &\frac{%
ab\left( b-a\right) }{2}\int_{0}^{1}\frac{\left\vert \left( 1-t\right)
^{\theta }-t^{\theta }\right\vert }{A_{t}^{2}}\left\vert f^{\prime }\left( 
\frac{ab}{A_{t}}\right) \right\vert dt  \notag \\
&\leq &\frac{ab\left( b-a\right) }{2}\int_{0}^{1}\frac{\left( 1-t\right)
^{\theta }+t^{\theta }}{A_{t}^{2}}\left\vert f^{\prime }\left( \frac{ab}{%
A_{t}}\right) \right\vert dt  \notag \\
&\leq &\frac{ab\left( b-a\right) }{2}\left( \int_{0}^{1}\frac{\left(
1-t\right) ^{\theta }+t^{\theta }}{A_{t}^{2}}dt\right) ^{1-1/q}  \notag \\
&&\times \left( \int_{0}^{1}\frac{\left( 1-t\right) ^{\theta }+t^{\theta }}{%
A_{t}^{2}}\left\vert f^{\prime }\left( \frac{ab}{A_{t}}\right) \right\vert
^{q}dt\right) ^{1/q}  \notag \\
&\leq &\frac{ab\left( b-a\right) }{2}\left( \int_{0}^{1}\frac{\left(
1-t\right) ^{\theta }+t^{\theta }}{A_{t}^{2}}dt\right) ^{1-1/q}  \notag \\
&&\times \left( 
\begin{array}{c}
\left( \int_{0}^{1}\frac{\left( 1-t\right) ^{\theta }+t^{\theta }}{A_{t}^{2}}%
t^{\alpha }dt\right) \left\vert f^{\prime }\left( a\right) \right\vert ^{q}
\\ 
+m\left( \int_{0}^{1}\frac{\left( 1-t\right) ^{\theta }+t^{\theta }}{%
A_{t}^{2}}\left( 1-t^{\alpha }\right) dt\right) \left\vert f^{\prime }\left(
b/m\right) \right\vert ^{q}%
\end{array}%
\right) ^{1/q}  \label{2.3}
\end{eqnarray}

calculating following integrals, we have%
\begin{eqnarray}
\int_{0}^{1}\frac{\left( 1-t\right) ^{\theta }+t^{\theta }}{A_{t}^{2}}dt
&=&\int_{0}^{1}\frac{u^{\theta }+\left( 1-u\right) ^{\theta }}{B_{u}^{2}}du 
\notag \\
&=&\frac{b^{-2}}{\theta +1}\left[ 
\begin{array}{c}
\begin{array}{c}
_{2}F_{1}\left( 2,\theta +1;\theta +2;1-\frac{a}{b}\right)%
\end{array}
\\ 
+%
\begin{array}{c}
_{2}F_{1}\left( 2,1;\theta +2;1-\frac{a}{b}\right)%
\end{array}%
\end{array}%
\right]  \notag \\
&=&C_{1}\left( \theta ;a,b\right)  \label{2.4}
\end{eqnarray}%
\begin{eqnarray}
\int_{0}^{1}\frac{\left( 1-t\right) ^{\theta }+t^{\theta }}{A_{t}^{2}}%
t^{\alpha }dt &=&\int_{0}^{1}\frac{u^{\theta }+\left( 1-u\right) ^{\theta }}{%
B_{u}^{2}}\left( 1-u\right) ^{\alpha }du  \notag \\
&=&\left[ 
\begin{array}{c}
\frac{\beta \left( \theta +1,\alpha +1\right) }{b^{2}}%
\begin{array}{c}
_{2}F_{1}\left( 2,\theta +1;\theta +\alpha +2;1-\frac{a}{b}\right)%
\end{array}
\\ 
+\frac{b^{-2}}{\theta +\alpha +1}%
\begin{array}{c}
_{2}F_{1}\left( 2,1;\theta +\alpha +2;1-\frac{a}{b}\right)%
\end{array}%
\end{array}%
\right]  \notag \\
&=&C_{2}\left( \theta ;\alpha ;a,b\right)  \label{2.5}
\end{eqnarray}%
\begin{eqnarray}
\int_{0}^{1}\frac{\left( 1-t\right) ^{\theta }+t^{\theta }}{A_{t}^{2}}\left(
1-t^{\alpha }\right) dt &=&\int_{0}^{1}\frac{\left( 1-t\right) ^{\theta
}+t^{\theta }}{A_{t}^{2}}dt-\int_{0}^{1}\frac{\left( 1-t\right) ^{\theta
}+t^{\theta }}{A_{t}^{2}}t^{\alpha }dt  \notag \\
&=&C_{1}\left( \theta ;a,b\right) -C_{2}\left( \theta ;\alpha ;a,b\right) 
\notag \\
&=&C_{3}\left( \theta ;\alpha ;a,b\right)  \label{2.6}
\end{eqnarray}

Thus, if we use $\left( \ref{2.4}\right) $-$\left( \ref{2.6}\right) $ in $%
\left( \ref{2.3}\right) $ we get the inequality of $\ \left( \ref{2.1}%
\right) $ and this completes the proof.\bigskip
\end{proof}
\end{theorem}

\begin{corollary}
In Theorem 5,

\begin{enumerate}
\item[(a)] If we take $\alpha =1$, $m=1$ we have the following inequality
for harmonically convex functions:%
\begin{eqnarray*}
\left\vert I_{f}\left( g;\theta ,a,b\right) \right\vert &\leq &\frac{%
ab\left( b-a\right) }{2}C_{1}^{1-1/q}\left( \theta ;a,b\right) \\
&&\times \left[ C_{2}\left( \theta ;1;a,b\right) \left\vert f^{\prime
}\left( a\right) \right\vert ^{q}+C_{3}\left( \theta ;1;a,b\right)
\left\vert f^{\prime }\left( b\right) \right\vert ^{q}\right] ^{1/q}\text{,}
\end{eqnarray*}

\item[(b)] If we take $\alpha =1$ we have the following inequality for
harmonically $m$-convex functions:%
\begin{eqnarray*}
\left\vert I_{f}\left( g;\theta ,a,b\right) \right\vert &\leq &\frac{%
ab\left( b-a\right) }{2}C_{1}^{1-1/q}\left( \theta ;a,b\right) \\
&&\times \left[ C_{2}\left( \theta ;1;a,b\right) \left\vert f^{\prime
}\left( a\right) \right\vert ^{q}+mC_{3}\left( \theta ;1;a,b\right)
\left\vert f^{\prime }\left( b/m\right) \right\vert ^{q}\right] ^{1/q}\text{,%
}
\end{eqnarray*}

\item[(c)] If we take $m=1$ we have the following inequality for
harmonically $\alpha $-convex functions:%
\begin{eqnarray*}
\left\vert I_{f}\left( g;\theta ,a,b\right) \right\vert &\leq &\frac{%
ab\left( b-a\right) }{2}C_{1}^{1-1/q}\left( \theta ;a,b\right) \\
&&\times \left[ C_{2}\left( \theta ;\alpha ;a,b\right) \left\vert f^{\prime
}\left( a\right) \right\vert ^{q}+C_{3}\left( \theta ;\alpha ;a,b\right)
\left\vert f^{\prime }\left( b\right) \right\vert ^{q}\right] ^{1/q}\text{.}
\end{eqnarray*}
\end{enumerate}
\end{corollary}

When $0<\theta \leq 1$, using Lemma 1 we shall give another result for
harmonically $\left( \alpha ,m\right) $-convex functions as follows:

\begin{theorem}
\bigskip Let $f:I\subset \left( 0,\infty \right) \rightarrow \mathbb{R}$ be
a differentiable function on $I{{}%
%TCIMACRO{\U{b0}}%
%BeginExpansion
{{}^\circ}%
%EndExpansion
}$, $a,b/m\in I{{}%
%TCIMACRO{\U{b0}}%
%BeginExpansion
{{}^\circ}%
%EndExpansion
}$ with $a<b$, $m\in \left( 0,1\right] $ and $f^{\prime }\in L\left[ a,b%
\right] $. If $\left\vert f^{\prime }\right\vert ^{q}$ is harmonically $%
\left( \alpha ,m\right) $-convex on $\left[ a,b/m\right] $ for some fixed $\
q\geq 1$,with $\alpha \in \left[ 0,1\right] $, then 
\begin{eqnarray}
\left\vert I_{f}\left( g;\theta ,a,b\right) \right\vert &\leq &\frac{%
ab\left( b-a\right) }{2}C_{4}^{1-1/q}\left( \theta ;a,b\right)  \notag \\
&&\times \left[ C_{5}\left( \theta ;\alpha ;a,b\right) \left\vert f^{\prime
}\left( a\right) \right\vert ^{q}+mC_{6}\left( \theta ;\alpha ;a,b\right)
\left\vert f^{\prime }\left( b/m\right) \right\vert ^{q}\right] ^{1/q}\text{
\ \ \ \ }  \label{2.7}
\end{eqnarray}%
where $0<\theta \leq 1$ and%
\begin{equation*}
C_{4}\left( \theta ;a,b\right) =\left[ 
\begin{array}{c}
\frac{b^{-2}}{\theta +1}%
\begin{array}{c}
_{2}F_{1}\left( 2,1;\theta +2;1-\frac{a}{b}\right)%
\end{array}
\\ 
-\frac{b^{-2}}{\theta +1}%
\begin{array}{c}
_{2}F_{1}\left( 2,\theta +1;\theta +2;1-\frac{a}{b}\right)%
\end{array}
\\ 
+\left( \frac{a+b}{2}\right) ^{-2}\frac{1}{\theta +1}%
\begin{array}{c}
_{2}F_{1}\left( 2,\theta +1;\theta +2;\frac{b-a}{b+a}\right)%
\end{array}%
\end{array}%
\right] \text{,}
\end{equation*}%
\begin{equation*}
C_{5}\left( \theta ;\alpha ;a,b\right) =\left[ 
\begin{array}{c}
\frac{b^{-2}}{\theta +\alpha +1}%
\begin{array}{c}
_{2}F_{1}\left( 2,1;\theta +\alpha +2;1-\frac{a}{b}\right)%
\end{array}
\\ 
-\frac{\beta \left( \theta +1,\alpha +1\right) }{b^{2}}%
\begin{array}{c}
_{2}F_{1}\left( 2,\theta +1;\theta +\alpha +2;1-\frac{a}{b}\right)%
\end{array}
\\ 
+\frac{\beta \left( \theta +1,\alpha +1\right) }{\left( a+b\right)
^{2}2^{\alpha -2}}%
\begin{array}{c}
_{2}F_{1}\left( 2,\theta +1;\theta +\alpha +2;\frac{b-a}{b+a}\right)%
\end{array}%
\end{array}%
\right] \text{,}
\end{equation*}
\end{theorem}

\begin{equation*}
C_{6}\left( \theta ;\alpha ;a,b\right) =C_{4}\left( \theta ;a,b\right)
-C_{5}\left( \theta ;\alpha ;a,b\right) \text{.}
\end{equation*}

\begin{proof}
\bigskip Let $A_{t}=tb+\left( 1-t\right) a$, $B_{u}=ua+\left( 1-u\right) b$.
From $\left( \ref{1.7}\right) $, using the power mean inequality and $\left( %
\ref{2.2}\right) $, we find%
\begin{eqnarray}
\left\vert I_{f}\left( g;\theta ,a,b\right) \right\vert &\leq &\frac{%
ab\left( b-a\right) }{2}\int_{0}^{1}\frac{\left\vert \left( 1-t\right)
^{\theta }-t^{\theta }\right\vert }{A_{t}^{2}}\left\vert f^{\prime }\left( 
\frac{ab}{A_{t}}\right) \right\vert dt  \notag \\
&\leq &\frac{ab\left( b-a\right) }{2}\left( \int_{0}^{1}\frac{\left\vert
\left( 1-t\right) ^{\theta }-t^{\theta }\right\vert }{A_{t}^{2}}dt\right)
^{1-1/q}  \notag \\
&&\times \left( \int_{0}^{1}\frac{\left\vert \left( 1-t\right) ^{\theta
}-t^{\theta }\right\vert }{A_{t}^{2}}\left\vert f^{\prime }\left( \frac{ab}{%
A_{t}}\right) \right\vert ^{q}dt\right) ^{1/q}  \notag \\
&\leq &\frac{ab\left( b-a\right) }{2}\left( \int_{0}^{1}\frac{\left\vert
\left( 1-t\right) ^{\theta }-t^{\theta }\right\vert }{A_{t}^{2}}dt\right)
^{1-1/q}  \notag \\
&&\times \left( 
\begin{array}{c}
\left( \int_{0}^{1}\frac{\left\vert \left( 1-t\right) ^{\theta }-t^{\theta
}\right\vert }{A_{t}^{2}}t^{\alpha }dt\right) \left\vert f^{\prime }\left(
a\right) \right\vert ^{q} \\ 
+m\left( \int_{0}^{1}\frac{\left\vert \left( 1-t\right) ^{\theta }-t^{\theta
}\right\vert }{A_{t}^{2}}\left( 1-t^{\alpha }\right) dt\right) \left\vert
f^{\prime }\left( b/m\right) \right\vert ^{q}%
\end{array}%
\right) ^{1/q}  \label{2.8}
\end{eqnarray}%
calculating following integrals by Lemma 1, we have%
\begin{eqnarray}
\int_{0}^{1}\frac{\left\vert \left( 1-t\right) ^{\theta }-t^{\theta
}\right\vert }{A_{t}^{2}}dt &=&\int_{0}^{1/2}\frac{\left( 1-t\right)
^{\theta }-t^{\theta }}{A_{t}^{2}}dt+\int_{1/2}^{1}\frac{t^{\theta }-\left(
1-t\right) ^{\theta }}{A_{t}^{2}}dt  \notag \\
&=&\int_{0}^{1}\frac{t^{\theta }-\left( 1-t\right) ^{\theta }}{A_{t}^{2}}%
dt+2\int_{0}^{1/2}\frac{\left( 1-t\right) ^{\theta }-t^{\theta }}{A_{t}^{2}}%
dt  \notag \\
&\leq &\int_{0}^{1}\frac{t^{\theta }}{A_{t}^{2}}dt-\int_{0}^{1}\frac{\left(
1-t\right) ^{\theta }}{A_{t}^{2}}dt+2\int_{0}^{1/2}\frac{\left( 1-2t\right)
^{\theta }}{A_{t}^{2}}dt  \notag \\
&=&\int_{0}^{1}\frac{\left( 1-u\right) ^{\theta }}{B_{u}^{2}}du-\int_{0}^{1}%
\frac{u^{\theta }}{B_{u}^{2}}du+\int_{0}^{1}\frac{\left( 1-u\right) ^{\theta
}}{\left( \frac{u}{2}b+\left( 1-\frac{u}{2}\right) a\right) ^{2}}du  \notag
\\
&=&\int_{0}^{1}\frac{\left( 1-u\right) ^{\theta }}{B_{u}^{2}}du-\int_{0}^{1}%
\frac{u^{\theta }}{B_{u}^{2}}du  \notag \\
&&+\int_{0}^{1}v^{\theta }\left( \frac{a+b}{2}\right) ^{-2}\left( 1-v\left( 
\frac{b-a}{b+a}\right) \right) ^{-2}dv  \notag \\
&=&\left[ 
\begin{array}{c}
\frac{b^{-2}}{\theta +1}%
\begin{array}{c}
_{2}F_{1}\left( 2,1;\theta +2;1-\frac{a}{b}\right)%
\end{array}
\\ 
-\frac{b^{-2}}{\theta +1}%
\begin{array}{c}
_{2}F_{1}\left( 2,\theta +1;\theta +2;1-\frac{a}{b}\right)%
\end{array}
\\ 
+\left( \frac{a+b}{2}\right) ^{-2}\frac{1}{\theta +1}%
\begin{array}{c}
_{2}F_{1}\left( 2,\theta +1;\theta +2;\frac{b-a}{b+a}\right)%
\end{array}%
\end{array}%
\right]  \notag \\
&=&C_{4}\left( \theta ;a,b\right)  \label{2.9}
\end{eqnarray}%
and similarly we get%
\begin{eqnarray}
\int_{0}^{1}\frac{\left\vert \left( 1-t\right) ^{\theta }-t^{\theta
}\right\vert }{A_{t}^{2}}t^{\alpha }dt &\leq &\int_{0}^{1}\frac{t^{\theta
+\alpha }}{A_{t}^{2}}dt-\int_{0}^{1}\frac{\left( 1-t\right) ^{\theta
}t^{\alpha }}{A_{t}^{2}}dt+2\int_{0}^{1/2}\frac{\left( 1-2t\right) ^{\theta
}t^{\alpha }}{A_{t}^{2}}dt  \notag \\
&=&\int_{0}^{1}\frac{\left( 1-u\right) ^{\theta +\alpha }}{B_{u}^{2}}%
du-\int_{0}^{1}\frac{u^{\theta }\left( 1-u\right) ^{\alpha }}{B_{u}^{2}}du 
\notag \\
&&+\int_{0}^{1}\frac{\left( 1-u\right) ^{\theta }\left( \frac{u}{2}\right)
^{\alpha }}{\left( \frac{u}{2}b+\left( 1-\frac{u}{2}\right) a\right) ^{2}}du
\notag \\
&=&\int_{0}^{1}\frac{\left( 1-u\right) ^{\theta +\alpha }}{B_{u}^{2}}%
du-\int_{0}^{1}\frac{u^{\theta }\left( 1-u\right) ^{\alpha }}{B_{u}^{2}}du 
\notag \\
&&+\frac{\left( \frac{a+b}{2}\right) ^{-2}}{2^{\alpha }}\int_{0}^{1}v^{%
\theta }\left( 1-v\right) ^{\alpha }\left( 1-v\left( \frac{b-a}{b+a}\right)
\right) ^{-2}dv  \notag \\
&=&\left[ 
\begin{array}{c}
\frac{b^{-2}}{\theta +\alpha +1}%
\begin{array}{c}
_{2}F_{1}\left( 2,1;\theta +\alpha +2;1-\frac{a}{b}\right)%
\end{array}
\\ 
-\frac{\beta \left( \theta +1,\alpha +1\right) }{b^{2}}%
\begin{array}{c}
_{2}F_{1}\left( 2,\theta +1;\theta +\alpha +2;1-\frac{a}{b}\right)%
\end{array}
\\ 
+\frac{\beta \left( \theta +1,\alpha +1\right) }{\left( a+b\right)
^{2}2^{\alpha -2}}%
\begin{array}{c}
_{2}F_{1}\left( 2,\theta +1;\theta +\alpha +2;\frac{b-a}{b+a}\right)%
\end{array}%
\end{array}%
\right]  \notag \\
&=&C_{5}\left( \theta ;\alpha ;a,b\right)  \label{2.10}
\end{eqnarray}%
\begin{eqnarray}
\int_{0}^{1}\frac{\left\vert \left( 1-t\right) ^{\theta }-t^{\theta
}\right\vert }{A_{t}^{2}}\left( 1-t^{\alpha }\right) dt &=&\int_{0}^{1}\frac{%
\left\vert \left( 1-t\right) ^{\theta }-t^{\theta }\right\vert }{A_{t}^{2}}%
dt-\int_{0}^{1}\frac{\left\vert \left( 1-t\right) ^{\theta }-t^{\theta
}\right\vert }{A_{t}^{2}}t^{\alpha }dt  \notag \\
&=&C_{4}\left( \theta ;a,b\right) -C_{5}\left( \theta ;\alpha ;a,b\right) 
\notag \\
&=&C_{6}\left( \theta ;\alpha ;a,b\right)  \label{2.11}
\end{eqnarray}%
Thus, if we use $\left( \ref{2.9}\right) $-$\left( \ref{2.11}\right) $ in $%
\left( \ref{2.8}\right) $ we get the inequality of $\ \left( \ref{2.7}%
\right) $ and this completes the proof.
\end{proof}

\begin{remark}
\bigskip If we take $\theta =1$, $\alpha =1$, $m=1$ in Theorem 6, then
inequality $\left( \ref{2.7}\right) $ becomes inequality $\left( \ref{1.3}%
\right) $ of Theorem 2.
\end{remark}

\begin{corollary}
In Theorem 6,

\begin{enumerate}
\item[(a)] If we take $\alpha =1$, $m=1$ we have the following inequality
for harmonically convex functions:%
\begin{eqnarray*}
\left\vert I_{f}\left( g;\theta ,a,b\right) \right\vert &\leq &\frac{%
ab\left( b-a\right) }{2}C_{4}^{1-1/q}\left( \theta ;a,b\right) \\
&&\times \left[ C_{5}\left( \theta ;1;a,b\right) \left\vert f^{\prime
}\left( a\right) \right\vert ^{q}+C_{6}\left( \theta ;1;a,b\right)
\left\vert f^{\prime }\left( b\right) \right\vert ^{q}\right] ^{1/q}\text{,}
\end{eqnarray*}

\item[(b)] If we take $\alpha =1$ we have the following inequality for
harmonically $m$-convex functions:%
\begin{eqnarray*}
\left\vert I_{f}\left( g;\theta ,a,b\right) \right\vert &\leq &\frac{%
ab\left( b-a\right) }{2}C_{4}^{1-1/q}\left( \theta ;a,b\right) \\
&&\times \left[ C_{5}\left( \theta ;1;a,b\right) \left\vert f^{\prime
}\left( a\right) \right\vert ^{q}+mC_{6}\left( \theta ;1;a,b\right)
\left\vert f^{\prime }\left( b/m\right) \right\vert ^{q}\right] ^{1/q}\text{,%
}
\end{eqnarray*}

\item[(c)] If we take $m=1$ we have the following inequality for
harmonically $\alpha $-convex functions:%
\begin{eqnarray*}
\left\vert I_{f}\left( g;\theta ,a,b\right) \right\vert &\leq &\frac{%
ab\left( b-a\right) }{2}C_{4}^{1-1/q}\left( \theta ;a,b\right) \\
&&\times \left[ C_{5}\left( \theta ;\alpha ;a,b\right) \left\vert f^{\prime
}\left( a\right) \right\vert ^{q}+C_{6}\left( \theta ;\alpha ;a,b\right)
\left\vert f^{\prime }\left( b\right) \right\vert ^{q}\right] ^{1/q}\text{.}
\end{eqnarray*}
\end{enumerate}
\end{corollary}

\begin{theorem}
Let $f:I\subset \left( 0,\infty \right) \rightarrow \mathbb{R}$ be a
differentiable function on $I{{}%
%TCIMACRO{\U{b0}}%
%BeginExpansion
{{}^\circ}%
%EndExpansion
}$, $a,b/m\in I{{}%
%TCIMACRO{\U{b0}}%
%BeginExpansion
{{}^\circ}%
%EndExpansion
}$ with $a<b$, $m\in \left( 0,1\right] $ and $f^{\prime }\in L\left[ a,b%
\right] $. If $\left\vert f^{\prime }\right\vert ^{q}$ is harmonically $%
\left( \alpha ,m\right) $-convex on $\left[ a,b/m\right] $ for some fixed $\
q>1$,with $\alpha \in \left[ 0,1\right] $, then 
\begin{eqnarray}
\left\vert I_{f}\left( g;\theta ,a,b\right) \right\vert &\leq &\frac{a\left(
b-a\right) }{2b}\left( \frac{1}{\theta p+1}\right) ^{1/p}\left( \frac{%
\left\vert f^{\prime }\left( a\right) \right\vert ^{q}+m\alpha \left\vert
f^{\prime }\left( b/m\right) \right\vert ^{q}}{\alpha +1}\right) ^{1/q} 
\notag \\
&&\times \left[ 
\begin{array}{c}
\begin{array}{c}
_{2}F_{1}^{1/p}\left( 2p,\theta p+1;\theta p+2;1-\frac{a}{b}\right)%
\end{array}
\\ 
+%
\begin{array}{c}
_{2}F_{1}^{1/p}\left( 2p,1;\theta p+2;1-\frac{a}{b}\right)%
\end{array}%
\end{array}%
\right] \text{ \ \ \ \ }  \label{2.12}
\end{eqnarray}%
where $\frac{1}{p}+\frac{1}{q}=1$.
\end{theorem}

\begin{proof}
Let $A_{t}=tb+\left( 1-t\right) a$, $B_{u}=ua+\left( 1-u\right) b$. From $%
\left( \ref{1.7}\right) $, using the property of the modulus, the H\"{o}lder
inequality and $\left( \ref{2.2}\right) $, we find%
\begin{eqnarray}
\left\vert I_{f}\left( g;\theta ,a,b\right) \right\vert &\leq &\frac{%
ab\left( b-a\right) }{2}\left[ \int_{0}^{1}\frac{\left( 1-t\right) ^{\theta }%
}{A_{t}^{2}}\left\vert f^{\prime }\left( \frac{ab}{A_{t}}\right) \right\vert
dt+\int_{0}^{1}\frac{t^{\theta }}{A_{t}^{2}}\left\vert f^{\prime }\left( 
\frac{ab}{A_{t}}\right) \right\vert dt\right]  \notag \\
&\leq &\frac{ab\left( b-a\right) }{2}\left[ 
\begin{array}{c}
\left( \int_{0}^{1}\frac{\left( 1-t\right) ^{\theta p}}{A_{t}^{2p}}dt\right)
^{1/p}\left( \int_{0}^{1}\left\vert f^{\prime }\left( \frac{ab}{A_{t}}%
\right) \right\vert ^{q}dt\right) ^{1/q} \\ 
+\left( \int_{0}^{1}\frac{t^{\theta p}}{A_{t}^{2p}}dt\right) ^{1/p}\left(
\int_{0}^{1}\left\vert f^{\prime }\left( \frac{ab}{A_{t}}\right) \right\vert
^{q}dt\right) ^{1/q}%
\end{array}%
\right]  \notag \\
&\leq &\frac{ab\left( b-a\right) }{2}\left( \left( \int_{0}^{1}\frac{%
u^{\theta p}}{B_{u}^{2p}}du\right) ^{1/p}+\left( \int_{0}^{1}\frac{\left(
1-u\right) ^{\theta p}}{B_{u}^{2p}}du\right) ^{1/p}\right)  \notag \\
&&\times \left( \int_{0}^{1}t^{\alpha }\left\vert f^{\prime }\left( a\right)
\right\vert ^{q}+m\left( 1-t^{\alpha }\right) \left\vert f^{\prime }\left(
b/m\right) \right\vert ^{q}dt\right) ^{1/q}  \notag \\
&=&\frac{ab\left( b-a\right) }{2}\left( K_{1}^{1/p}+K_{2}^{1/p}\right)
\left( \frac{\left\vert f^{\prime }\left( a\right) \right\vert ^{q}+m\alpha
\left\vert f^{\prime }\left( b/m\right) \right\vert ^{q}}{\alpha +1}\right)
^{1/q}.  \label{2.13}
\end{eqnarray}%
calculating $K_{1}$ and $K_{2}$ we have 
\begin{equation}
K_{1}=\int_{0}^{1}\frac{u^{\theta p}}{B_{u}^{2p}}du=\frac{b^{-2p}}{\theta p+1%
}%
\begin{array}{c}
_{2}F_{1}\left( 2p,\theta p+1;\theta p+2;1-\frac{a}{b}\right)%
\end{array}
\label{2.14}
\end{equation}%
\begin{equation}
K_{1}=\int_{0}^{1}\frac{\left( 1-u\right) ^{\theta p}}{B_{u}^{2p}}du=\frac{%
b^{-2p}}{\theta p+1}%
\begin{array}{c}
_{2}F_{1}\left( 2p,1;\theta p+2;1-\frac{a}{b}\right)%
\end{array}
\label{2.15}
\end{equation}%
\bigskip Thus, if we use $\left( \ref{2.14}\right) $,$\left( \ref{2.15}%
\right) $ in $\left( \ref{2.13}\right) $ we get the inequality of $\ \left( %
\ref{2.12}\right) $ and this completes the proof.
\end{proof}

\begin{corollary}
In Theorem 7,

\begin{enumerate}
\item[(a)] If we take $\alpha =1$, $m=1$ we have the following inequality
for harmonically convex functions:%
\begin{eqnarray*}
\left\vert I_{f}\left( g;\theta ,a,b\right) \right\vert &\leq &\frac{a\left(
b-a\right) }{2b}\left( \frac{1}{\theta p+1}\right) ^{1/p}\left( \frac{%
\left\vert f^{\prime }\left( a\right) \right\vert ^{q}+\left\vert f^{\prime
}\left( b\right) \right\vert ^{q}}{2}\right) ^{1/q} \\
&&\times \left[ 
\begin{array}{c}
\begin{array}{c}
_{2}F_{1}^{1/p}\left( 2p,\theta p+1;\theta p+2;1-\frac{a}{b}\right)%
\end{array}
\\ 
+%
\begin{array}{c}
_{2}F_{1}^{1/p}\left( 2p,1;\theta p+2;1-\frac{a}{b}\right)%
\end{array}%
\end{array}%
\right] \text{,}
\end{eqnarray*}

\item[(b)] If we take $\alpha =1$ we have the following inequality for
harmonically $m$-convex functions:%
\begin{eqnarray*}
\left\vert I_{f}\left( g;\theta ,a,b\right) \right\vert &\leq &\frac{a\left(
b-a\right) }{2b}\left( \frac{1}{\theta p+1}\right) ^{1/p}\left( \frac{%
\left\vert f^{\prime }\left( a\right) \right\vert ^{q}+m\left\vert f^{\prime
}\left( b/m\right) \right\vert ^{q}}{2}\right) ^{1/q} \\
&&\times \left[ 
\begin{array}{c}
\begin{array}{c}
_{2}F_{1}^{1/p}\left( 2p,\theta p+1;\theta p+2;1-\frac{a}{b}\right)%
\end{array}
\\ 
+%
\begin{array}{c}
_{2}F_{1}^{1/p}\left( 2p,1;\theta p+2;1-\frac{a}{b}\right)%
\end{array}%
\end{array}%
\right] \text{,}
\end{eqnarray*}

\item[(c)] If we take $m=1$ we have the following inequality for
harmonically $\alpha $-convex functions:%
\begin{eqnarray*}
\left\vert I_{f}\left( g;\theta ,a,b\right) \right\vert &\leq &\frac{a\left(
b-a\right) }{2b}\left( \frac{1}{\theta p+1}\right) ^{1/p}\left( \frac{%
\left\vert f^{\prime }\left( a\right) \right\vert ^{q}+\alpha \left\vert
f^{\prime }\left( b\right) \right\vert ^{q}}{\alpha +1}\right) ^{1/q} \\
&&\times \left[ 
\begin{array}{c}
\begin{array}{c}
_{2}F_{1}^{1/p}\left( 2p,\theta p+1;\theta p+2;1-\frac{a}{b}\right)%
\end{array}
\\ 
+%
\begin{array}{c}
_{2}F_{1}^{1/p}\left( 2p,1;\theta p+2;1-\frac{a}{b}\right)%
\end{array}%
\end{array}%
\right] \text{.}
\end{eqnarray*}
\end{enumerate}
\end{corollary}

\begin{theorem}
\bigskip Let $f:I\subset \left( 0,\infty \right) \rightarrow \mathbb{R}$ be
a differentiable function on $I{{}%
%TCIMACRO{\U{b0}}%
%BeginExpansion
{{}^\circ}%
%EndExpansion
}$, $a,b/m\in I{{}%
%TCIMACRO{\U{b0}}%
%BeginExpansion
{{}^\circ}%
%EndExpansion
}$ with $a<b$, $m\in \left( 0,1\right] $ and $f^{\prime }\in L\left[ a,b%
\right] $. If $\left\vert f^{\prime }\right\vert ^{q}$ is harmonically $%
\left( \alpha ,m\right) $-convex on $\left[ a,b/m\right] $ for some fixed $\
q>1$,with $\alpha \in \left[ 0,1\right] $, then 
\begin{eqnarray}
\left\vert I_{f}\left( g;\theta ,a,b\right) \right\vert &\leq &\frac{a\left(
b-a\right) }{2b}\left( \frac{1}{\theta p+1}\right) ^{1/p}\left( \frac{1}{%
\alpha +1}\right) ^{1/q}  \notag \\
&&\times \left( 
\begin{array}{c}
\begin{array}{c}
_{2}F_{1}\left( 2q,1;\alpha +2;1-\frac{a}{b}\right)%
\end{array}%
\left\vert f^{\prime }\left( a\right) \right\vert ^{q} \\ 
+m\left[ 
\begin{array}{c}
\left( \alpha +1\right) 
\begin{array}{c}
_{2}F_{1}\left( 2q,1;2;1-\frac{a}{b}\right)%
\end{array}
\\ 
-%
\begin{array}{c}
_{2}F_{1}\left( 2q,1;\alpha +2;1-\frac{a}{b}\right)%
\end{array}%
\end{array}%
\right] \left\vert f^{\prime }\left( b/m\right) \right\vert ^{q}%
\end{array}%
\right) ^{1/q}\text{ \ \ \ \ }  \label{2.16}
\end{eqnarray}%
where $\frac{1}{p}+\frac{1}{q}=1$.
\end{theorem}

\begin{proof}
Let $A_{t}=tb+\left( 1-t\right) a$, $B_{u}=ua+\left( 1-u\right) b$. From $%
\left( \ref{1.7}\right) $, using the H\"{o}lder inequality, Lemma 1, and $%
\left( \ref{2.2}\right) $, we find%
\begin{eqnarray}
\left\vert I_{f}\left( g;\theta ,a,b\right) \right\vert &\leq &\frac{%
ab\left( b-a\right) }{2}\int_{0}^{1}\frac{\left\vert \left( 1-t\right)
^{\theta }-t^{\theta }\right\vert }{A_{t}^{2}}\left\vert f^{\prime }\left( 
\frac{ab}{A_{t}}\right) \right\vert dt  \notag \\
&\leq &\frac{ab\left( b-a\right) }{2}\left( \int_{0}^{1}\left\vert \left(
1-t\right) ^{\theta }-t^{\theta }\right\vert ^{p}dt\right) ^{1/p}  \notag \\
&&\times \left( \int_{0}^{1}\frac{1}{A_{t}^{2q}}\left\vert f^{\prime }\left( 
\frac{ab}{A_{t}}\right) \right\vert ^{q}dt\right) ^{1/q}  \notag \\
&\leq &\frac{ab\left( b-a\right) }{2}\left( \int_{0}^{1}\left\vert
1-2t\right\vert ^{\theta p}dt\right) ^{1/p}  \notag \\
&&\times \left( \int_{0}^{1}\frac{1}{A_{t}^{2q}}\left[ t^{\alpha }\left\vert
f^{\prime }\left( a\right) \right\vert ^{q}+m\left( 1-t^{\alpha }\right)
\left\vert f^{\prime }\left( b/m\right) \right\vert ^{q}\right] dt\right)
^{1/q}  \label{2.17}
\end{eqnarray}%
calculating following integrals, we have%
\begin{equation}
\int_{0}^{1}\left\vert 1-2t\right\vert ^{\theta p}dt=\frac{1}{\theta p+1}
\label{2.18}
\end{equation}%
\begin{equation}
\int_{0}^{1}\frac{t^{\alpha }}{A_{t}^{2q}}dt=\int_{0}^{1}\frac{\left(
1-u\right) ^{\alpha }}{B_{u}^{2q}}dt=\frac{b^{-2q}}{\alpha +1}%
\begin{array}{c}
_{2}F_{1}\left( 2q,1;\alpha +2;1-\frac{a}{b}\right)%
\end{array}
\label{2.19}
\end{equation}%
\begin{eqnarray}
\int_{0}^{1}\frac{1-t^{\alpha }}{A_{t}^{2q}}dt &=&\int_{0}^{1}\frac{1-\left(
1-u\right) ^{\alpha }}{B_{u}^{2q}}dt=b^{-2q}%
\begin{array}{c}
_{2}F_{1}\left( 2q,1;2;1-\frac{a}{b}\right)%
\end{array}
\notag \\
&&-\frac{b^{-2q}}{\alpha +1}%
\begin{array}{c}
_{2}F_{1}\left( 2q,1;\alpha +2;1-\frac{a}{b}\right)%
\end{array}%
\text{ \ \ \ \ \ \ \ \ \ \ \ \ }  \label{2.20}
\end{eqnarray}%
Thus, if we use $\left( \ref{2.18}\right) $-$\left( \ref{2.20}\right) $ in $%
\left( \ref{2.17}\right) $ we get the inequality of $\ \left( \ref{2.16}%
\right) $ and this completes the proof.
\end{proof}

\begin{remark}
\bigskip \bigskip If we take $\theta =1$, $\alpha =1$, $m=1$ in Theorem 8,
then inequality $\left( \ref{2.16}\right) $ becomes inequality $\left( \ref%
{1.4}\right) $ of Theorem 3.
\end{remark}

\begin{corollary}
In Theorem 8,

\begin{enumerate}
\item[(a)] If we take $\alpha =1$, $m=1$ we have the following inequality
for harmonically convex functions:%
\begin{eqnarray*}
\left\vert I_{f}\left( g;\theta ,a,b\right) \right\vert &\leq &\frac{a\left(
b-a\right) }{2^{1+1/q}b}\left( \frac{1}{\theta p+1}\right) ^{1/p} \\
&&\times \left( 
\begin{array}{c}
\begin{array}{c}
_{2}F_{1}\left( 2q,1;3;1-\frac{a}{b}\right)%
\end{array}%
\left\vert f^{\prime }\left( a\right) \right\vert ^{q} \\ 
+\left[ 
\begin{array}{c}
2%
\begin{array}{c}
_{2}F_{1}\left( 2q,1;2;1-\frac{a}{b}\right)%
\end{array}
\\ 
-%
\begin{array}{c}
_{2}F_{1}\left( 2q,1;3;1-\frac{a}{b}\right)%
\end{array}%
\end{array}%
\right] \left\vert f^{\prime }\left( b\right) \right\vert ^{q}%
\end{array}%
\right) ^{1/q}\text{,}
\end{eqnarray*}

\item[(b)] If we take $\alpha =1$ we have the following inequality for
harmonically $m$-convex functions:%
\begin{eqnarray*}
\left\vert I_{f}\left( g;\theta ,a,b\right) \right\vert &\leq &\frac{a\left(
b-a\right) }{2^{1+1/q}b}\left( \frac{1}{\theta p+1}\right) ^{1/p} \\
&&\times \left( 
\begin{array}{c}
\begin{array}{c}
_{2}F_{1}\left( 2q,1;3;1-\frac{a}{b}\right)%
\end{array}%
\left\vert f^{\prime }\left( a\right) \right\vert ^{q} \\ 
+m\left[ 
\begin{array}{c}
2%
\begin{array}{c}
_{2}F_{1}\left( 2q,1;2;1-\frac{a}{b}\right)%
\end{array}
\\ 
-%
\begin{array}{c}
_{2}F_{1}\left( 2q,1;3;1-\frac{a}{b}\right)%
\end{array}%
\end{array}%
\right] \left\vert f^{\prime }\left( b/m\right) \right\vert ^{q}%
\end{array}%
\right) ^{1/q}\text{,}
\end{eqnarray*}

\item[(c)] If we take $m=1$ we have the following inequality for
harmonically $\alpha $-convex functions:%
\begin{eqnarray*}
\left\vert I_{f}\left( g;\theta ,a,b\right) \right\vert &\leq &\frac{a\left(
b-a\right) }{2b}\left( \frac{1}{\theta p+1}\right) ^{1/p}\left( \frac{1}{%
\alpha +1}\right) ^{1/q} \\
&&\times \left( 
\begin{array}{c}
\begin{array}{c}
_{2}F_{1}\left( 2q,1;\alpha +2;1-\frac{a}{b}\right)%
\end{array}%
\left\vert f^{\prime }\left( a\right) \right\vert ^{q} \\ 
+\left[ 
\begin{array}{c}
\left( \alpha +1\right) 
\begin{array}{c}
_{2}F_{1}\left( 2q,1;2;1-\frac{a}{b}\right)%
\end{array}
\\ 
-%
\begin{array}{c}
_{2}F_{1}\left( 2q,1;\alpha +2;1-\frac{a}{b}\right)%
\end{array}%
\end{array}%
\right] \left\vert f^{\prime }\left( b\right) \right\vert ^{q}%
\end{array}%
\right) ^{1/q}\text{.}
\end{eqnarray*}
\end{enumerate}
\end{corollary}

\begin{theorem}
\bigskip Let $f:I\subset \left( 0,\infty \right) \rightarrow \mathbb{R}$ be
a differentiable function on $I{{}%
%TCIMACRO{\U{b0}}%
%BeginExpansion
{{}^\circ}%
%EndExpansion
}$, $a,b/m\in I{{}%
%TCIMACRO{\U{b0}}%
%BeginExpansion
{{}^\circ}%
%EndExpansion
}$ with $a<b$, $m\in \left( 0,1\right] $ and $f^{\prime }\in L\left[ a,b%
\right] $. If $\left\vert f^{\prime }\right\vert ^{q}$ is harmonically $%
\left( \alpha ,m\right) $-convex on $\left[ a,b/m\right] $ for some fixed $\
q>1$,with $\alpha \in \left[ 0,1\right] $, then 
\begin{eqnarray}
\left\vert I_{f}\left( g;\theta ,a,b\right) \right\vert &\leq &\frac{%
ab\left( b-a\right) }{2^{\left( 1/p\right) -1}\left( a+b\right) ^{2}}\left( 
\frac{1}{\theta p+1}\right) ^{1/p}\left( \frac{\left\vert f^{\prime }\left(
a\right) \right\vert ^{q}+m\alpha \left\vert f^{\prime }\left( b/m\right)
\right\vert ^{q}}{\alpha +1}\right) ^{1/q}  \notag \\
&&\times \left[ 
\begin{array}{c}
\begin{array}{c}
_{2}F_{1}\left( 2p,\theta p+1;\theta p+2;\frac{b-a}{b+a}\right)%
\end{array}
\\ 
+%
\begin{array}{c}
_{2}F_{1}\left( 2p,\theta p+1;\theta p+2;\frac{a-b}{b+a}\right)%
\end{array}%
\end{array}%
\right] ^{1/p}  \label{2.21}
\end{eqnarray}%
where $\frac{1}{p}+\frac{1}{q}=1$.
\end{theorem}

\begin{proof}
Let $A_{t}=tb+\left( 1-t\right) a$. From $\left( \ref{1.7}\right) $, using
the H\"{o}lder inequality, Lemma 1, and $\left( \ref{2.2}\right) $, we find%
\begin{eqnarray}
\left\vert I_{f}\left( g;\theta ,a,b\right) \right\vert &\leq &\frac{%
ab\left( b-a\right) }{2}\int_{0}^{1}\frac{\left\vert \left( 1-t\right)
^{\theta }-t^{\theta }\right\vert }{A_{t}^{2}}\left\vert f^{\prime }\left( 
\frac{ab}{A_{t}}\right) \right\vert dt  \notag \\
&\leq &\frac{ab\left( b-a\right) }{2}\int_{0}^{1}\frac{\left\vert
1-2t\right\vert ^{\theta }}{A_{t}^{2}}\left\vert f^{\prime }\left( \frac{ab}{%
A_{t}}\right) \right\vert dt  \notag \\
&\leq &\frac{ab\left( b-a\right) }{2}\left( \int_{0}^{1}\frac{\left\vert
1-2t\right\vert ^{\theta p}}{A_{t}^{2p}}dt\right) ^{1/p}\left(
\int_{0}^{1}\left\vert f^{\prime }\left( \frac{ab}{A_{t}}\right) \right\vert
^{q}dt\right) ^{1/q}  \notag \\
&\leq &\frac{ab\left( b-a\right) }{2}\left( \int_{0}^{1}\frac{\left\vert
1-2t\right\vert ^{\theta p}}{A_{t}^{2p}}dt\right) ^{1/p}  \notag \\
&&\times \left( \int_{0}^{1}t^{\alpha }\left\vert f^{\prime }\left( a\right)
\right\vert ^{q}+m\left( 1-t^{\alpha }\right) \left\vert f^{\prime }\left(
b/m\right) \right\vert ^{q}dt\right) ^{1/q}  \notag \\
&=&\frac{ab\left( b-a\right) }{2}\left[ \int_{0}^{1/2}\frac{\left(
1-2t\right) ^{\theta p}}{A_{t}^{2p}}dt+\int_{1/2}^{1}\frac{\left(
2t-1\right) ^{\theta p}}{A_{t}^{2p}}dt\right] ^{1/p}  \notag \\
&&\times \left( \frac{\left\vert f^{\prime }\left( a\right) \right\vert
^{q}+m\alpha \left\vert f^{\prime }\left( b/m\right) \right\vert ^{q}}{%
\alpha +1}\right) ^{1/q}  \label{2.22}
\end{eqnarray}%
calculating following integrals, we have%
\begin{eqnarray}
\int_{0}^{1/2}\frac{\left( 1-2t\right) ^{\theta p}}{A_{t}^{2p}}dt &=&\frac{1%
}{2}\int_{0}^{1}\frac{\left( 1-u\right) ^{\theta p}}{\left( \frac{u}{2}%
b+\left( 1-\frac{u}{2}\right) a\right) ^{2p}}du  \notag \\
&=&\frac{\left( a+b\right) ^{-2p}}{2^{1-2p}}\int_{0}^{1}v^{\theta p}\left(
1-v\left( \frac{b-a}{b+a}\right) \right) ^{-2p}dv  \notag \\
&=&\frac{\left( a+b\right) ^{-2p}}{2^{1-2p}\left( \theta p+1\right) }%
\begin{array}{c}
_{2}F_{1}\left( 2p,\theta p+1;\theta p+2;\frac{b-a}{b+a}\right)%
\end{array}
\label{2.23}
\end{eqnarray}%
\begin{eqnarray}
\int_{1/2}^{1}\frac{\left( 2t-1\right) ^{\theta p}}{A_{t}^{2p}}dt &=&\frac{1%
}{2}\int_{1}^{2}\frac{\left( u-1\right) ^{\theta p}}{\left( \frac{u}{2}%
b+\left( 1-\frac{u}{2}\right) a\right) ^{2p}}du  \notag \\
&=&\frac{\left( a+b\right) ^{-2p}}{2^{1-2p}}\int_{0}^{1}v^{\theta p}\left(
1-v\left( \frac{a-b}{b+a}\right) \right) ^{-2p}dv  \notag \\
&=&\frac{\left( a+b\right) ^{-2p}}{2^{1-2p}\left( \theta p+1\right) }%
\begin{array}{c}
_{2}F_{1}\left( 2p,\theta p+1;\theta p+2;\frac{a-b}{b+a}\right)%
\end{array}
\label{2.24}
\end{eqnarray}%
Thus, if we use $\left( \ref{2.23}\right) $-$\left( \ref{2.24}\right) $ in $%
\left( \ref{2.22}\right) $ we get the inequality of $\ \left( \ref{2.21}%
\right) $ and this completesF the proof.
\end{proof}

\begin{corollary}
In Theorem 9,

\begin{enumerate}
\item[(a)] If we take $\alpha =1$, $m=1$ we have the following inequality
for harmonically convex functions:%
\begin{eqnarray*}
\left\vert I_{f}\left( g;\theta ,a,b\right) \right\vert &\leq &\frac{%
ab\left( b-a\right) }{2^{\left( 1/p\right) -1}\left( a+b\right) ^{2}}\left( 
\frac{1}{\theta p+1}\right) ^{1/p}\left( \frac{\left\vert f^{\prime }\left(
a\right) \right\vert ^{q}+\left\vert f^{\prime }\left( b\right) \right\vert
^{q}}{2}\right) ^{1/q} \\
&&\times \left[ 
\begin{array}{c}
\begin{array}{c}
_{2}F_{1}\left( 2p,\theta p+1;\theta p+2;\frac{b-a}{b+a}\right)%
\end{array}
\\ 
+%
\begin{array}{c}
_{2}F_{1}\left( 2p,\theta p+1;\theta p+2;\frac{a-b}{b+a}\right)%
\end{array}%
\end{array}%
\right] ^{1/p}\text{,}
\end{eqnarray*}

\item[(b)] If we take $\alpha =1$ we have the following inequality for
harmonically $m$-convex functions:%
\begin{eqnarray*}
\left\vert I_{f}\left( g;\theta ,a,b\right) \right\vert &\leq &\frac{%
ab\left( b-a\right) }{2^{\left( 1/p\right) -1}\left( a+b\right) ^{2}}\left( 
\frac{1}{\theta p+1}\right) ^{1/p}\left( \frac{\left\vert f^{\prime }\left(
a\right) \right\vert ^{q}+m\left\vert f^{\prime }\left( b/m\right)
\right\vert ^{q}}{2}\right) ^{1/q} \\
&&\times \left[ 
\begin{array}{c}
\begin{array}{c}
_{2}F_{1}\left( 2p,\theta p+1;\theta p+2;\frac{b-a}{b+a}\right)%
\end{array}
\\ 
+%
\begin{array}{c}
_{2}F_{1}\left( 2p,\theta p+1;\theta p+2;\frac{a-b}{b+a}\right)%
\end{array}%
\end{array}%
\right] ^{1/p}\text{,}
\end{eqnarray*}

\item[(c)] If we take $m=1$ we have the following inequality for
harmonically $\alpha $-convex functions:%
\begin{eqnarray*}
\left\vert I_{f}\left( g;\theta ,a,b\right) \right\vert &\leq &\frac{%
ab\left( b-a\right) }{2^{\left( 1/p\right) -1}\left( a+b\right) ^{2}}\left( 
\frac{1}{\theta p+1}\right) ^{1/p}\left( \frac{\left\vert f^{\prime }\left(
a\right) \right\vert ^{q}+\alpha \left\vert f^{\prime }\left( b\right)
\right\vert ^{q}}{\alpha +1}\right) ^{1/q} \\
&&\times \left[ 
\begin{array}{c}
\begin{array}{c}
_{2}F_{1}\left( 2p,\theta p+1;\theta p+2;\frac{b-a}{b+a}\right)%
\end{array}
\\ 
+%
\begin{array}{c}
_{2}F_{1}\left( 2p,\theta p+1;\theta p+2;\frac{a-b}{b+a}\right)%
\end{array}%
\end{array}%
\right] ^{1/p}\text{.}
\end{eqnarray*}
\end{enumerate}
\end{corollary}

\end{document}